\documentclass[12pt,psamsfonts, draft]{amsart}
\usepackage{amssymb,amsmath}
\usepackage{amsthm}
\usepackage[mathscr]{eucal} 
\usepackage{enumerate, indentfirst}
\usepackage[fleqn,tbtags]{mathtools}
\usepackage{color}
\usepackage{courier}
\usepackage{hyperref}
\usepackage[a4paper,left=2.5cm,right=2.5cm,top=3.2cm,bottom=3.2cm]{geometry}
\usepackage{multicol}
\usepackage{xr}

\newtheorem{theorem}{Theorem}[section]

\newtheorem{lemma}[theorem]{Lemma}

\newtheorem{corollary}[theorem]{Corollary}

\renewcommand{\phi}{\varphi}
\renewcommand{\theta}{\vartheta}

\newcommand{\Kk}{\mathsf{K}}
\newcommand{\Ll}{\mathsf{L}}
\newcommand{\Jj}{\mathsf{J}}

\DeclarePairedDelimiterX\sip[2]{(}{)}{#1\,\delimsize\vert\,#2}
\DeclarePairedDelimiterX\siptilde[2]{(}{)_{\!_{\widetilde{A}}}}{#1\,\delimsize\vert\,#2}
\DeclarePairedDelimiterX\sipn[2]{(}{)_{\nu}}{#1\,\delimsize\vert\,#2}
\DeclarePairedDelimiterX\sipm[2]{(}{)_{\mu}}{#1\,\delimsize\vert\,#2}
\DeclarePairedDelimiterX\set[2]{\{}{\}}{#1\,\delimsize\vert\,#2}
\DeclarePairedDelimiterX\dual[2]{\langle}{\rangle}{#1,#2}
\DeclarePairedDelimiterX\sipa[2]{(}{)_{\!_A}}{#1\,\delimsize\vert\,#2}
\DeclarePairedDelimiterX\sipb[2]{(}{)_{\!_B}}{#1\,\delimsize\vert\,#2}

\newcommand{\Ee}{\mathfrak{E}}
\newcommand{\Ees}{\mathfrak{E}^{\ast}}

\newcommand{\dtw} {\mathbf{D}_{\tf}\wf}
\newcommand{\dwt}{\mathbf{D}_{\wf}\tf}

\newcommand{\uf}{\mathfrak{u}}
\newcommand{\vf}{\mathfrak{v}}
\newcommand{\tf} {\mathfrak{t}}
\newcommand{\wf} {\mathfrak{w}}
\newcommand{\ssf} {\mathfrak{s}}
\newcommand{\Xx}{\mathfrak{X}}

\newcommand{\fpx}{\mathcal{F}_+(\Xx)}
\newcommand{\kpx}{\mathcal{K}_+(\Xx)}

\renewcommand*{\phi}{\varphi}
\renewcommand*{\theta}{\vartheta}
\newcommand{\tkerw}{\tf_{{}_{\ker\wf}}}

\DeclareMathOperator{\tform}{\mathfrak{t}}

\begin{document}

\title{Positive definite operator functions and sesquilinear forms}

\author{Tam\'as Titkos}
\address{Tam\'as Titkos, Institute of Mathematics, E\"otv\"os Lor\'and University, P\'azm\'any P\'eter s\'et\'any 1/c, H-1117, Budapest, Hungary; }

\email{\textrm{titkos@cs.elte.hu}}

\dedicatory{Dedicated to Zolt\'an Sebesty\'en on the occasion of his $70$th birthday}

\maketitle
\begin{abstract}
Due to the fundamental works of T. Ando, W. Szyma\'nski, F. H. Szafraniec, and many others it is well known that sesquilinear forms play an important role in dilation theory. The crucial fact is that every positive definite operator function induces a sesquilinear form in a natural way. The purpose of this survey-like paper is to apply some recent results of Z. Sebesty\'en, Zs. Tarcsay, and the author for such functions. While most of the results are not new, the paper's main contribution is the unified discussion from the viewpoint of sesquilinear forms.

\end{abstract}

\section{Sesquilinear forms}\label{sect1}
In this preliminary section we review first some of the standard notions and notations and give a brief survey of some recent results needed throughout the paper. We focus on the decomposition and Radon--Nikodym theory of nonnnegative sesquilinear forms that we will apply on positive definite operator functions in Section \ref{sect2}. Our main references are \cite[Section 2]{lebdec} and \cite{stt1}.

\subsection{Notions, notations}

Let $\Xx$ be a complex linear space and let $\tf$ be a nonnegative sesquilinear form (or shortly just \emph{form}) on it. That is, $\tf$ is a mapping from
$\Xx\times\Xx$ to $\mathbb{C}$, which is linear in the first argument, antilinear in the second argument, and the corresponding
quadratic form
\begin{align*}
\forall x\in\Xx:\quad\tf[x]:=\tf(x,x)
\end{align*}
is nonnegative. A crucial fact is that a form is uniquely determined by its quadratic form due to the polarization formula
\begin{align*}
\forall x,y\in\Xx:\quad\tf(x,y)=\frac{1}{4}\sum\limits_{k=0}^3i^k\tf[x+i^ky].
\end{align*}

The set of forms will be denoted by $\fpx$.
For $\tf,\wf\in\fpx$ we write $\tf\leq\wf$ if $\tf[x]\leq\wf[x]$ for all $x\in\Xx$. \emph{Domination} means that there exists a constant $c$ such that $\tf\leq c\cdot\wf$. Using the ordering we can define singularity and almost domination. The forms $\tf$ and $\wf$ are \emph{singular} ($\tf\perp\wf$) if for every form $\ssf$ the inequalities $\ssf\leq\tf$ and $\ssf\leq\wf$ imply that $\ssf=\mathfrak{0}$ (i.e., $\ssf$ is the identically zero form). We say that $\tf$ is \emph{almost dominated by $\wf$} (in symbols: $\tf\ll_{\mathrm{ad}}\wf$) if there exists a monotonically nondecreasing sequence of forms $\tf_n$, each dominated by
$\wf$, such that $\tf=\sup\limits_{n\in\mathbb{N}}\tf_n$ (pointwise supremum). 

Now, we define two  important notions, which are motivated by classical measure theory. Let $\tf$ and $\wf$ be forms, $\tf$ is called \emph{absolutely continuous with respect to $\wf$} (or $\tf$ is $\wf$-absolutely continuous, in symbols: $\tf\ll_\mathrm{ac}\wf$), if $\wf[x]=0$ implies $\tf[x]=0$ for all $x\in\Xx$. We say that $\tf$ is \emph{strongly $\wf$-absolutely continuous} ($\tf\ll_{\mathrm{s}}\wf$, in symbols), if
\begin{align*}
\forall(x_n)_{n\in\mathbb{N}}\in\Xx^{\mathbb{N}}:\quad\big((\tf[x_n-x_m]\to0)~\wedge~(\wf[x_n]\to0)\big)~\Longrightarrow~\tf[x_n]\to0.
\end{align*}
Remark that this notion is called \emph{closability} in \cite{lebdec}; cf. also \cite{simon} and \cite{kato}. The following theorem says that strong absolute continuity is closely related to the ordering. For the proof see \cite[Theorem 3.8]{lebdec}.

\begin{theorem}
Let $\tf$ and $\wf$ be forms on $\Xx$. Then $\tf$ is almost dominated by $\wf$ if and only if $\tf$ is strongly $\wf$-absolutely continuous.
\end{theorem}

It is important to mention that if $\tf\in\fpx$ then the square root of its quadratic form defines a seminorm on $\Xx$. Hence the set
\begin{align*}
\ker\tf:=\big\{x\in\Xx~\big|~\tf[x]=0\big\}
\end{align*}
is a linear subspace of $\Xx$. The Hilbert space $\mathcal{H}_{\tf}$ denotes the completion of the inner product space $\Xx/_{\ker\tf}$ equipped with the natural inner product 
\begin{align*}
\forall x,y\in\Xx:\quad(x+\ker\tf~|~y+\ker\tf)_\tf:=\tf(x,y).
\end{align*}
Observe that $\tf$ is $\wf$-absolutely continuous if and only if the canonical embedding (which assigns the coset $x+\ker\tf$ to $x+\ker\wf$) from $\mathcal{H}_{\wf}$ to $\mathcal{H}_{\tf}$ is well-defined. Strong absolute continuity means that this embedding is a closable operator.

We close this subsection with a Radon--Nikodym-type result. This was proved independently by Zs. Tarcsay, from a different point of view. For more background we refer the reader to \cite{trn}. 
\begin{lemma}
Let $\tf$ and $\wf$ be forms on $\Xx$ and assume that $\tf\leq c\cdot\wf$ for some $c>0$. Then for every $y\in\Xx$ there exists a unique vector $\xi_y$ in $\mathcal{H}_{\wf}$ such that
\begin{align*}
\forall x\in\Xx:\qquad\tf(x,y)=(x+\ker\wf~|~\xi_y)_{\wf}.
\end{align*}
\end{lemma}
\begin{proof}
Let $y$ be an arbitrary but fixed element of $\Xx$ and define the linear functional $\Phi_y$ as follows
\begin{align*}
\Phi_y:\Xx/_{\ker\wf}\to\mathbb{C};\qquad x+\ker\wf\mapsto(x+\ker\tf~|~y+\ker\tf)_{\tf}.
\end{align*}
According to the Cauchy-Schwarz inequality and the assumption it is clear that $\Phi_y$ is a bounded linear functional. Indeed,
\begin{align*}
|\Phi_y(x+\ker\wf)|^2\leq\|x+\ker\tf\|_{\tf}^2\cdot\|y+\ker\tf\|_{\tf}^2\leq c^2\cdot\|x+\ker\wf\|_{\wf}^2\cdot\|y+\ker\wf\|_{\wf}^2.
\end{align*}
Consequently, due to the Riesz representation theorem there exists a unique vector $\xi_y$ in $\mathcal{H}_{\wf}$ such that
\begin{align*}
\forall x\in\Xx:\qquad\tf(x,y)=(x+\ker\tf~|~y+\ker\tf)_{\tf}=\Phi_y(x+\ker\tf)=(x+\ker\wf~|~\xi_y)_{\wf}.
\end{align*}
\end{proof}
\begin{theorem}\label{RNF}
Let $\tf,\wf\in\fpx$ be forms on $\Xx$ and let $\tf$ be almost dominated by $\wf$. Then for every $y\in\Xx$ there exists a sequence $(y_n)_{n\in\mathbb{N}}\in\Xx^{\mathbb{N}}$ such that
\begin{align*}
\forall x\in\Xx:\qquad\tf(x,y)=\lim\limits_{n\to +\infty}\wf(x,y_n).
\end{align*}
\end{theorem}
\begin{proof}
Fix an arbitrary $y\in\Xx$. Since $\tf$ is almost dominated by $\wf$, there exists a suitable sequence $(\tf_n)_{n\in\mathbb{N}}$ of $\wf$-dominated forms and a sequence $(\xi_{y,n})_{n\in\mathbb{N}}$ of representant vectors such that
\begin{align*}
\lim\limits_{n\to +\infty}\tf_n=\tf\qquad\mbox{and}\qquad(\forall x\in\Xx)~(\forall n\in\mathbb{N}):\quad\tf_n(x,y)=(x+\ker\wf~|~\xi_{y,n})_{\wf}.
\end{align*}
As $\tf_n\leq\tf$, we can apply the Cauchy--Schwarz inequality on the form $\tf-\tf_n$ that gives 
\begin{align*}
\big|\big(\tf-\tf_n\big)(x,y)\big|^2\leq\big(\tf-\tf_n\big)[x]\big(\tf-\tf_n\big)[y]\to 0,\qquad n\to +\infty, 
\end{align*}
whence we infer that 
\begin{equation*}
\tf(x,y)=\lim\limits_{n\to +\infty} t_n(x,y)=\lim\limits_{n\to +\infty}(x+\ker \wf~|~\xi_{y,n})_\wf.
\end{equation*}
Since $\Xx/_{\ker\wf}$ is dense in $\mathcal{H}_{\wf}$ we can choose a sequence $(y_n)_{n\in\mathbb{N}}\in\Xx^{\mathbb{N}}$ such that
\begin{align*}
\big\|\xi_{y,n}-(y_n+\ker\wf)\big\|_{\wf}\to 0.
\end{align*}
According to the Cauchy--Schwarz inequality, this implies that
\begin{align*}
\big|(x+\ker\wf~|~\xi_{y,n})_{\wf}-(x+\ker\wf~|~y_n+\ker\wf)_{\wf}\big|\to 0
\end{align*}
and thus
\begin{align*}
\forall x\in\Xx:\qquad\tf(x,y)=\lim\limits_{n\to +\infty}\wf(x,y_n).
\end{align*}
\end{proof}

\subsection{Decomposition theorems}
In this subsection we recall two basic results of decomposition theory of forms. The first one is the so-called short-type decomposition, which is a decomposition of $\tf$ into absolutely continuous and singular parts. The key notion is the \emph{short} of a form to a linear subspace of $\Xx$ (for the details see \cite{stt2}), which is a generalization of the well known concept of operator short \cite{anderson, krein, pekarev ext}.

Let $\tf$ and $\wf$ be forms on $\Xx$, then the short of $\tf$ to the subspace $\ker\wf$ is defined by
\begin{align*}
\forall x\in\Xx:\quad\tf_{{}_{\ker\wf}}[x]:=\inf\limits_{y\in\ker\wf}\tf[x-y].
\end{align*}
The short-type decomposition theorem is stated as follows (\cite[Theorem 1.2]{stt2}).
\begin{theorem} \label{STD}
Let $\tf,\wf\in\fpx$ be forms on $\Xx$. Then there exists a short-type decomposition of $\tf$ with respect to $\wf$. Namely,
\begin{align*}
\tf=\tkerw+(\tf-\tkerw),
\end{align*}
where the first summand is $\wf$-absolutely continuous and the second one is $\wf$-singular.
Furthermore, $\tkerw$ is the largest element of the set
\begin{align*}
\big\{\ssf\in\fpx~\big|~(\ssf\leq\tf)~\wedge~(\ssf\ll_{\mathrm{ac}}\wf) \big\}.
\end{align*}
The decomposition is unique precisely when $\tkerw$ is dominated by $\wf$.
\end{theorem}

A decomposition of $\tf$ into strongly $\wf$-absolutely continuous (or $\wf$-almost dominated) and $\wf$-singular parts is called Lebesgue-type decomposition. This is a generalization of the well-known decomposition result of T. Ando \cite{ando1976} (see also \cite{tarcsay}). The existence of such a decomposition for forms was proved first by Hassi, Sebesty\'en, and de Snoo in \cite{lebdec}. In order to present their result we need to introduce the notion of parallel sum. The parallel sum $\tf:\wf$ of the forms $\tf$ and $\wf$ is determined by the formula
\begin{align*}
\forall x\in\Xx:\quad(\tf:\wf)[x]:=\inf_{y\in\Xx}\big\{\tf[x-y]+\wf[y]\big\}.
    \end{align*}
We will see that the form 
\begin{align*}
\dwt:=\sup\limits_{n\in\mathbb{N}}(\tf:n\wf)
\end{align*}
plays an important role in this paper. (For the properties of parallel addition and the operator $\mathbf{D}$ see \cite[Proposition 2.3, Lemma 2.4]{lebdec}.) 

Primarily, $\dwt$ is the so-called almost dominated part of $\tf$ with respect to $\wf$, as the following fundamental theorem states \cite[Theorem 2.11]{lebdec} (see also \cite[Theorem 2.3]{stt1} and \cite[Theorem 3]{strn}).

\begin{theorem}\label{LDT}
Let $\tf$ and $\wf$ be forms on $\Xx$. Then the decomposition
\begin{align*}
\tf=\dwt+(\tf-\dwt)
\end{align*}
is a Lebesgue-type decomposition of $\tf$ with respect to $\wf$. That is, $\dwt$ is almost dominated by $\wf$, $(\tf-\dwt)$ is $\wf$-singular. Furthermore,
$\dwt$ is the largest element of the set
\begin{align*}
\big\{\ssf\in\fpx~\big|~(\ssf\leq\tf)~\wedge~(\ssf\ll_{\mathrm{ad}}\wf) \big\}.
\end{align*}
The decomposition is unique precisely when $\dwt$ is dominated by $\wf$.
\end{theorem}

Moreover, for the almost dominated part we have the following two formulae (for the proofs see \cite[Lemma 2.2, Theorem 2.3]{stt1} and \cite[Theorem 2.7]{stt1}):
\begin{align*}
(\dwt)[x]=\inf\big\{\lim_{n\rightarrow +\infty}\tform[x-x_n]~\big|~(x_n)_{n\in\mathbb{N}}\in\Xx^{\mathbb{N}}:(\tf[x_n-x_m]\to 0)~\wedge~(\wf[x_n]\to 0)\big\}
\end{align*}
and
\begin{align*}
(\dwt)[x]=\inf\big\{\liminf_{n\rightarrow +\infty}\tform[x-x_n]~\big|~(x_n)_{n\in\mathbb{N}}\in\Xx^{\mathbb{N}}:~\wf[x_n]\to 0\big\}.
\end{align*} 
The first interesting observation is \cite[Theorem 1.5]{ldocvnf}, which implies for example for every finite measures $\mu$ and $\nu$ that the $\nu$-absolutely continuous part of $\mu$ is absolutely continuous with respect to the $\mu$-absolutely continuous part of $\nu$ \cite[Theorem 3.5 (b)]{ldocvnf}. An analogous result regarding representable functionals can be found in \cite{tarcsay-funk}. We present here a proof which is simpler than the original one in \cite{ldocvnf}.

\begin{theorem}\label{abszfolytabszfolyt}
Let $\tf$ and $\wf$ be forms on $\Xx$, and consider their Lebesgue-type decompositions with respect to each other. Then the almost dominated parts are mutually almost dominated, i.e.,
\begin{align*}
\dwt\ll_{\mathrm{ad}}\dtw\qquad\mbox{and}\qquad\dtw\ll_{\mathrm{ad}}\dwt.
\end{align*}
\end{theorem}
\begin{proof}
Observe first that if $\uf_1$, $\uf_2$, and $\vf$ are forms such that $\uf_1\leq\uf_2$ and $\uf_1\ll_{\mathrm{ad}}\vf$, then \begin{align*}\uf_1\ll_{\mathrm{ad}}\mathbf{D}_{\uf_2}\vf.\end{align*}
Indeed, if $\uf_1$ is almost dominated by $\vf$, then there exists a monotonically nondecreasing sequence of forms $(\uf_{1,n})_{n\in\mathbb{N}}$ such that $\sup\limits_{n\in\mathbb{N}}\uf_{1,n}=\uf$ and $\uf_{1,n}$ is dominated by $\vf$ for all $n\in\mathbb{N}$ (i.e., $\uf_{1,n}\leq c_n\vf$ for some $c_n\geq 0$). Consequently,
\begin{align*}
\uf_{1,n}=\mathbf{D}_{\uf_2}\uf_{1,n}\leq\mathbf{D}_{\uf_2}c_n\vf=c_n\mathbf{D}_{\uf_2}\vf
\end{align*}
which means that $\uf_1\ll_{\mathrm{ad}}\mathbf{D}_{\uf_2}\vf$. Now, apply the previous observation with $\uf_1:=\dwt$, $\uf_2=\tf$, and $\vf=\wf$.
\end{proof}

\subsection{Order structure and some extremal problems}\label{os}
This subsection is devoted to investigating the connection between some order properties of $\big(\fpx,\leq\big)$ and the Lebesgue type decomposition.

The first natural question is whether the infimum (i.e., the greatest lower bound) $\tf\wedge\wf$ of $\tf$ and $\wf$ exists in $\fpx$. The infimum problem has a long history in the theory of Hilbert space operators. Kadison proved that the set of bounded self-adjoint operators is a so-called anti-lattice \cite{kadison}. For bounded positive operators the infimum problem was proved by Moreland and Gudder provided the space is finite dimensional \cite{gudder}. 

The general case was solved by T. Ando in \cite{ando1999}. He showed that the infimum of two positive operators
$A$ and $B$ exists in the positive cone if and only if the generalized shorts (for this notion see \cite{ando1976}) $[B]A$ and $[A]B$ are comparable. An analogous result concerning forms was given in \cite{titkosinf}.

Recall that the infimum of $\tf$ and $\wf$ exists if there is a form denoted by $\tf\wedge\wf$, for which $\tf\wedge\wf\leq\tf$, $\tf\wedge\wf\leq\wf$, and the inequalities $\uf\leq\tf$ and $\uf\leq\wf$ imply that $\uf\leq\tf\wedge\wf$.

\begin{theorem}\label{forminf}
Let $\tf,\wf\in\fpx$ be forms on $\Xx$. Then the following statements are equivalent.
    \begin{itemize}
        \item[(i)] $\dtw\leq\dwt$ or $\dwt\leq\dtw$.
        \item[(ii)] $\dtw\leq\tf$ or $\dwt\leq\wf$.
        \item[(iii)] The infimum $\tf\wedge\wf$ exists.
    \end{itemize}
\end{theorem}
A nonzero form $\tf$ is called \emph{minimal}, if $\wf\leq\tf$ implies that $\lambda\tf=\wf$ for some $\lambda\geq 0$. Or equivalently (see \cite[Theorem 5.10]{psonf}), for every form $\wf$ there exists a $\lambda\geq0$ such that $\dtw=\lambda\tf$.
\begin{corollary}
Let $\tf$ be a minimal form on $\Xx$. Then for every $\wf\in\fpx$ the infimum $\tf\wedge\wf$ exists.
\end{corollary}

The Lebesgue decomposition theory of forms is encountered again by examining the extremal points of the convex set $[0,\tf]$. Here the segment $[\tf_1,\tf_2]$ for $\tf_1,\tf_2\fpx,~\tf_1\leq\tf_2$ is defined to be the convex set 
\begin{align*}
[\tf_1,\tf_2]=\big\{\ssf\in\fpx~\big|~\tf_1\leq\ssf\leq\tf_2\big\}.
\end{align*}
The following theorems characterize the extremal points of form segments; for the proofs and other references see \cite[Theorem 11]{cof} and \cite[Section 5]{psonf}.

\begin{theorem}\label{disjoint part}
Let $\uf$ and $\tf$ be forms on $\Xx$, such that $\uf\leq\tf$. The following statements are equivalent
\begin{itemize}
\item[$(i)$] $\uf$ and $\tf-\uf$ are singular,
\item[$(ii)$] $\mathbf{D}_{\uf}\tf=\uf$,
\item[$(iii)$] $\uf$ is an extreme point of the convex set $[0,\tf]$.
\end{itemize}
\end{theorem}

\begin{theorem}\label{ex(t,t+w) substeq ex(0,t+w)}
 Let $\tf$ and $\wf$ be forms on $\Xx$. Then the following statements are equivalent
\begin{itemize}
\item[$(i)$] $\tf$ is an extreme point of $[0,\tf+\wf]$,
\item[$(ii)$] $\mathrm{ex}[\tf,\tf+\wf]\subseteq\mathrm{ex}[0,\tf+\wf]$.
\end{itemize}
Replacing $\wf$ with $\wf-\tf$ (if $\tf\leq\wf$) we have
\begin{align*}
\tf\in\mathrm{ex}[0,\wf]\hspace{0.3cm}\Leftrightarrow\hspace{0.3cm}\mathrm{ex}[\tf,\wf]\subseteq\mathrm{ex}[0,\wf].
\end{align*}
\end{theorem}

\section{Positive definite operator functions}\label{sect2}
In this section we carry over the previous theorems for positive definite operator functions. Szyma\'nski in \cite{s1} presented a general dilation theory governed by forms. We will see (after making some generalities) that the absolutely continuous part in Theorem \ref{STD} (and the almost dominated part in Theorem \ref{LDT}) is the largest dilatable part in some sense. Finally, we describe some order properties of kernels. Throughout this section we will use the notations of \cite[Section 7]{lebdec}, which is our main reference. Recall again that almost domination and strong absolute continuity (or closability) are equivalent concepts for forms.

Let $S$ be a non-empty set, and let $\Ee$ be a complex Banach space (with topological dual $\Ee^{\ast}$). The dual pairing of $x\in\Ee$ and $x^{\ast}\in\Ee^{\ast}$ is denoted by $\langle x,x^{\ast}\rangle$. Here the mapping
\begin{align*}
\langle\cdot,\cdot\rangle:\Ee\times\Ees\to\mathbb{C}
\end{align*}
is linear in its first, conjugate linear in its second variable. The Banach space of bounded linear operators from $\Ee$ to $\Ees$ will be denoted by $\mathbf{B}(\Ee,\Ees)$. 

Let $\Xx$ be the complex linear space of functions on $S$ with values in $\mathbf{B}(\Ee,\Ees)$ with finite support. We say that the function
\begin{align*}
\Kk:S\times S\to \mathbf{B}(\Ee,\Ees)
\end{align*}
is a \emph{positive definite operator function}, or shortly a \emph{kernel} on $S$ if
\begin{align*}
\forall f\in\Xx:\qquad\sum\limits_{s,t\in S}\langle f(t),\Kk(s,t)f(s)\rangle\geq 0.
\end{align*}
We associate a form with $\Kk$ by setting
\begin{align*}
\forall f,g\in\Xx:\qquad\wf_{\Kk}(f,g):=\sum\limits_{s,t\in S}\langle f(t),\Kk(s,t)g(s)\rangle.
\end{align*}
The set of kernels will be denoted by $\kpx$. If $\Kk$ and $\Ll$ are kernels, we write $\Kk\prec\Ll$ if $\wf_\Kk\leq\wf_\Ll$.

The following lemma states that the order structures of forms and of kernels are the same. Here  we give just an outline, for the complete proof see \cite[Lemma 7.1]{lebdec}. An analogous result in context of bounded positive operators can be found in \cite[(2.2) Theorem]{as}. 

\begin{lemma}
Let $\Kk\in\kpx$ be a kernel on $S$ with associated form $\wf_\Kk$ and let $\wf$ be a form on $\Xx$. Then the following statements are equivalent
\begin{itemize}
\item[$(i)$] $\wf\leq\wf_\Kk$,
\item[$(ii)$] $\wf=\wf_\Ll$ for a unique kernel $\Ll\prec\Kk$.
\end{itemize}
\end{lemma}
\begin{proof}
Implication $(ii)\Rightarrow(i)$ follows from the definitions. To prove the converse implication define for each $s\in S$ and $x\in\Ee$ the function
\begin{align*}
h_{s,x}\in\Xx; \qquad \forall u\in S: \quad h_{s,x}(u):=\delta_s(u)x
\end{align*}
where $\delta_s$ is the Dirac function concentrated to $s$. Now, define $\Ll$ pointwise as follows. For each $s,t\in S$
\begin{align*}
\forall x,y\in\Ee:\qquad\langle x,\Ll(s,t)y\rangle:=\wf(h_{t,x},h_{s,y})
\end{align*}
It follows from the nonnegativity of $\wf[\,\cdot\,]$ that 
\begin{align*}\sum\limits_{s,t\in S}\langle f(t),\Ll(s,t)f(s)\rangle 
\end{align*}
is nonnegative for all $f\in\Xx$. The only thing we need is to show that $L(s,t)\in\mathbf{B}(\Ee,\Ees)$. According to the Cauchy-Schwarz inequality, we have for all $x,y\in\Ee$ that
\begin{align*}
\begin{split}
|\langle x,L(s,t)y\rangle|^2=|\wf(h_{t,x},h_{s,y})|^2&\leq\wf[h_{t,x}]\cdot\wf[h_{s,y}]\leq\wf_\Kk[h_{t,x}]\cdot\wf_\Kk[h_{s,y}]
\\
=\langle x,K(t,t)x \rangle\cdot\langle y,K(s,s)y \rangle&\leq\|K(t,t)\|_{\mathbf{B}(\Ee,\Ees)}\cdot\|K(s,s)\|_{\mathbf{B}(\Ee,\Ees)}\cdot\|x\|_\Ee^2\cdot\|y\|_\Ee^2.
\end{split}
\end{align*}
\end{proof}
We emphasize here that the preceding is the key observation of this section. Most of the results gathered below are immediate consequences of this lemma, and the theorems listed in Section \ref{sect1}.

Now, we can define domination, almost domination, singularity, closability, and (strong) absolute continuity of kernels via their associated forms. We say that \emph{$\Kk$ is $\Ll$-almost dominated; $\Ll$-closable; (strongly)-$\Ll$-absolutely continuous} if $\wf_\Kk$ is $\wf_\Ll$-almost dominated; $\wf_\Ll$-closable; (strongly)-$\wf_\Ll$-absolutely continuous, respectively. \emph{$\Kk$ and $\Ll$ are singular} if $\wf_\Kk$ and $\wf_\Ll$ are singular.

Before stating the short-type and Lebesgue-type decomposition of kernels, we mention a result of W. Szyma\'nski (reduced to our less general setting). For the details we refer the reader to \cite[(3.5) Theorem]{s1}.

\begin{theorem}\label{dilat}
Let $\Kk,\Ll\in\kpx$ be kernels on $S$ with associated forms $\wf_\Kk$ and $\wf_\Ll$. Then
\begin{itemize}
\item[$(a)$]
$\Kk$ is absolutely continuous with respect to $\Ll$ (i.e., $\ker\wf_\Ll\subseteq\ker\wf_\Kk$) if and only if there exists a Hilbert space $\mathcal{H}$ and a linear mapping $T:\Xx/_{\ker\wf_\Ll}\to\mathcal{H}$ such that
\begin{align*}
\langle y,K(s,t)x\rangle=\big(T(h_{t,y}+\ker\wf_\Ll)~\big|~T(h_{s,x}+\ker\wf_\Ll)\big)_{\mathcal{H}},
\end{align*}
\item[$(b)$] $\Kk$ is strongly absolutely continuous with respect to $\Ll$ (i.e., $\wf_\Kk$ is strongly $\wf_\Ll$-absolutely continuous) if and only if 
there exists a Hilbert space $\mathcal{H}$ and a closed linear mapping $T:\Xx/_{\ker\wf_\Ll}\to\mathcal{H}$ such that
\begin{align*}
\langle y,K(s,t)x\rangle=\big(T(h_{t,y}+\ker\wf_\Ll)~\big|~T(h_{s,x}+\ker\wf_\Ll)\big)_{\mathcal{H}}.
\end{align*}
\end{itemize}
\end{theorem}
The operator $T$ is called the \emph{dilation} of $\Kk$ and the auxiliary space $\mathcal{H}$ is called the dilation space.

In view of the previous theorem, the following two decomposition theorems can be stated as follows. For every pair of kernels $\Kk$ and $\Ll$ there is a maximal part of $\Kk$ which has a (closed) dilation with respect to $\Ll$. These are straightforward consequences of Theorem \ref{STD} and of Theorem \ref{LDT}.

\begin{theorem}
Let $\Kk,\Ll\in\kpx$ be kernels on $S$. Then there exists a short-type decomposition of $\Kk$ with respect to $\Ll$, i.e., the first summand is $\Ll$-absolutely continuous and the second one is $\Ll$-singular. Namely
\begin{align*}
\Kk=\Kk_{\mathrm{ac},\Ll}+\Kk_{\mathrm{s},\Ll},
\end{align*}
where
\begin{align*}
\sum\limits_{s,t\in S}\langle f(t),\Kk_{\mathrm{ac},\Ll}(s,t)f(s)\rangle=\inf\limits_{g\in\ker\wf_{\Ll}}\sum\limits_{s,t\in S}\langle f(t)-g(t),\Kk(s,t)(f(s)-g(s))\rangle.
\end{align*}
The decomposition is unique precisely when $\Kk_{\mathrm{ac},\Ll}$ is dominated by $\Ll$.
\end{theorem}

\begin{theorem}
Let $\Kk,\Ll\in\kpx$ be kernels on $S$. Then the decomposition
\begin{align*}
\Kk=\mathbf{D}_\Ll\Kk+(\Kk-\mathbf{D}_\Ll\Kk),
\end{align*}
is a Lebesgue-type decomposition of $\Kk$ with respect to $\Ll$. That is, $\mathbf{D}_\Ll\Kk$ is strongly $\Ll$-absolutely continuous, $(\Kk-\mathbf{D}_\Ll\Kk)$ is $\Ll$-singular. The almost dominated part $\mathbf{D}_\Ll\Kk$ is defined by
\begin{align*}
\wf_{\mathbf{D}_\Ll\Kk}:=\mathbf{D}_{\wf_\Ll}\wf_\Kk,
\end{align*}
and hence 
\begin{align*}
\wf_{\mathbf{D}_\Ll\Kk}[f]=\inf\big\{\lim_{n\rightarrow +\infty}\wf_\Kk[f-g_n]~\big|~(g_n)_{n\in\mathbb{N}}\in\Xx^{\mathbb{N}}:(\wf_\Kk[g_n-g_m]\to 0)~\wedge~(\wf_\Ll[g_n]\to 0)\big\}
\end{align*}
and
\begin{align*}
\wf_{\mathbf{D}_\Ll\Kk}[f]=\inf\big\{\liminf_{n\rightarrow +\infty}\wf_\Kk[f-g_n]~\big|~(g_n)_{n\in\mathbb{N}}\in\Xx^{\mathbb{N}}:~\wf_\Ll[x_n]\to 0\big\}.
\end{align*} 
The decomposition is unique precisely when $\mathbf{D}_\Ll\Kk$ is dominated by $\Ll$.
\end{theorem}
Due to Theorem \ref{RNF} we have the following Radon--Nikodym-type result for kernels.
\begin{corollary}
Let $\Kk,\Ll\in\kpx$ be kernels on $S$ and assume that $\Kk$ is almost dominated by $\Ll$. Then for every $g\in\Xx$ there exists a sequence $(g_n)_{n\in\mathbb{N}}\in\Xx^{\mathbb{N}}$ such that
\begin{align*}
\forall f\in\Xx:\quad \sum\limits_{s,t\in S}\langle f(t),\Kk(s,t)g(s)\rangle=\lim\limits_{n\to +\infty}\sum\limits_{s,t\in S}\langle f(t),\Ll(s,t)g_n(s)\rangle.
\end{align*}
\end{corollary}
The following statements are immediate consequences of Theorem \ref{abszfolytabszfolyt} and Theorem \ref{forminf}.
\begin{corollary}
Let $\Kk,\Ll\in\kpx$ be kernels on $S$, then $\mathbf{D}_{\Ll}\Kk$ is $\mathbf{D}_{\Kk}\Ll$-almost dominated. And by symmetry, $\mathbf{D}_{\Kk}\Ll$ is $\mathbf{D}_{\Ll}\Kk$-almost dominated.
\end{corollary}
\begin{corollary} Let $\Kk$ and $\Ll$ be kernels on $S$. Then the infimum $\Kk\wedge\Ll$ of $\Kk$ and $\Ll$ exists precisely when $\mathbf{D}_\Kk\Ll$ and $\mathbf{D}_\Ll\Kk$ are comparable. 
\end{corollary}
Finally, we have the following characterizations according to Theorem \ref{disjoint part} and Theorem \ref{ex(t,t+w) substeq ex(0,t+w)}.
\begin{corollary}
Let $\Jj,\Kk\in\kpx$ be kernels on $S$, such that $\Jj\prec\Kk$. The following statements are equivalent.
\begin{itemize}
\item[$(i)$] $\Jj$ and $\Kk-\Jj$ are singular.
\item[$(ii)$] $\mathbf{D}_{\Jj}\Kk=\Jj$.
\item[$(iii)$] $\Jj$ is an extreme point of the convex set $[\mathsf{0},\Kk]=\big\{\mathsf{U}\in\kpx~\big|~\mathsf{0}\prec\mathsf{U}\prec\Kk\big\}$.
\end{itemize}
\end{corollary}
In view of Theorem \ref{dilat} the previous corollary says that the extremal points of the convex set $[\mathsf{0},\Kk]$ are precisely those kernels that have closed dilation.
\begin{corollary}
 Let $\Kk,\Ll\in\kpx$ be kernels on $S$. Then the following statements are equivalent
\begin{itemize}
\item[$(i)$] $\Kk$ is an extreme point of $[\mathsf{0},\Kk+\Ll]$.
\item[$(ii)$] $\mathrm{ex}[\Kk,\Kk+\Ll]\subseteq\mathrm{ex}[\mathsf{0},\Kk+\Ll]$.
\end{itemize}
Replacing $\Ll$ with $\Ll-\Kk$ (if $\Kk\prec\Ll$) we have
\begin{align*}
\Kk\in\mathrm{ex}[\mathsf{0},\Ll]\hspace{0.3cm}\Leftrightarrow\hspace{0.3cm}\mathrm{ex}[\Kk,\Ll]\subseteq\mathrm{ex}[\mathsf{0},\Ll].
\end{align*}
\end{corollary}
\newpage
\begin{center}
\textbf{Acknowledgement}
\end{center}
I am greatly indebted to Professor Zolt\'an Sebesty\'en for his valuable guidance, encouragement, and support throughout my study.


\begin{thebibliography}{1}



\bibitem{anderson}
    \newblock{Anderson, Jr., W. N. and Trapp, G. E.},
    \newblock{\em Shorted operators II.},
    \newblock{SIAM J. Appl. Math.}, {28} {(1975)}, {60--71.}

\bibitem{ando1976}
\newblock{Ando,~T.},
\newblock {\em Lebesgue-type decomposition of positive operators},
\newblock {Actha Sci. Math. (Szeged)} {38} {(1976)}, {253--260}.

\bibitem{ando1999}
\newblock{Ando,~T.},
\newblock{\em Problem of infimum in the positive cone},
 \newblock{Analytic and geometric inequalities and applications},
 \newblock{Math. Appl., {478}, 1--12, Kluwer Acad. Publ., {Dordrecht}, {1999}}.

\bibitem{as}
\newblock{Ando,~T. and Szyma\'nski,~W.},
\newblock{\em Order Structure and Lebesgue Decomposition of Positive Definite Operator Functions},
\newblock{Indiana
Univ. Math. J.}, \newblock{35 (1986)}, \newblock{157–-173.}

\bibitem{gudder}
\newblock{Moreland,~T. and Gudder,~S.},
\newblock {\em Infima of Hilbert space effects},
\newblock {Linear Algebra and its Applications} {286} {(1999)}, {1--17}.

\bibitem{krein}
\newblock{Krein,~M. G.},
\newblock{\em The theory of self-adjoint extensions of semi-bounded Hermitian operators},
\newblock{[Mat. Sbornik] 10 (1947), 431--495}.

\bibitem{lebdec}
\newblock{Hassi,~S. and Sebesty{\'e}n,~Z. and de Snoo,~H.},
\newblock{\em Lebesgue type decompositions for nonnegative forms},
\newblock{{J. Funct. Anal.} {257}{(12)} {(2009)}, {3858--3894.}}

\bibitem{kadison}
\newblock{Kadison,~R.},
\newblock {\em Order properties of bounded self-adjoint operators},
\newblock {Proc. Amer. Math. Soc.} {34} {(1951)}, {505--510}.

\bibitem{kato}
\newblock{Kato,~T.},
\newblock{\em Perturbation Theory for Linear Operators}, \newblock{Springer-Verlag (Berlin, 1980).}

\bibitem{pekarev ext}
\newblock{Pekarev, E. L.},
\newblock{\em Shorts of operators and some extremal problems},
\newblock{Acta Sci. Math. (Szeged)}, {56} {(1992)}, {147--163.}



\bibitem{stt1}
\newblock{Sebesty\'en,~Z. and Tarcsay,~Zs. and Titkos,~T.},
\newblock{\em Lebesgue decomposition theorems},
\newblock{Acta Sci. Math. (Szeged)}, {79(1-2)} {(2013)}, {219--233.}

\bibitem{stt2}
\newblock{Sebesty\'en,~Z. and Tarcsay,~Zs. and Titkos,~T.},
\newblock{\em A Short-type Decomposition Of Forms},
\newblock{Manuscript (http://arxiv.org/pdf/1406.6635.pdf)}.

\bibitem{cof}
\newblock{Sebesty\'en,~Z. and Titkos,~T.},
\newblock{\em Complement of forms},
\newblock{Positivity}, {17} {(2013)}, {1--15.}


\bibitem{strn}
\newblock{Sebesty\'en,~Z. and Titkos,~T.},
\newblock{\em A Radon--Nikodym type theorem for forms},
\newblock{Positivity}, {17} {(2013)}, {863--873.}


\bibitem{psonf}
\newblock{Sebesty\'en,~Z. and Titkos,~T.},
\newblock{\em Parallel subtraction of nonnegative forms},
\newblock{Acta Math. Hung.}, \newblock{136(4) (2012)}, \newblock{252--269.}

\bibitem{simon}
\newblock{Simon,~B.},
\newblock {\em A canonical decomposition for quadratic forms with applications to monotone convergence theorems,}
\newblock {J. Funct. Anal.} \newblock{ 28 (1978), 377--385.}


\bibitem{s1}
\newblock{Szyma\'nski,~W.},
\newblock{\em Positive forms and dilations}, \newblock{Trans Amer Math.} \newblock{301(2)} \newblock{(1987)}, \newblock{761--780.}


\bibitem{sz}
\newblock{Szafraniec,~F.H.},
\newblock{\em Boundedness of the shift operator related to positive definite forms: An application to moment problems},
\newblock{Ark. Mat.} \newblock{19 (1981), no. 2}, \newblock{251--259}.

\bibitem{trn}
\newblock{Tarcsay,~Zs.},
\newblock{\em Radon--Nikodym theorems for nonnegative forms, measures and representable functionals},
\newblock{Manuscript (http://arxiv.org/pdf/1403.5891v2.pdf).}

\bibitem{tarcsay}
\newblock{Tarcsay, Zs.},
\newblock{\em Lebesgue-type decomposition of positive operators},
\newblock{Positivity}, {17(3)} {(2013)}, {803--817.}

\bibitem{tarcsay-funk}
\newblock{Tarcsay, Zs.},
\newblock{\em Lebesgue decomposition for representable functionals on ${}^{\ast}$-algebras},
\newblock{Glasgow Mathematical Journal
}, {to appear.}


\bibitem{titkosinf}
\newblock{Titkos, T.},
\newblock{\em Ando's theorem for nonnegative forms},
\newblock{Positivity 16(2012)}, \newblock{619--626}.

\bibitem{ldocvnf}
\newblock{Titkos, T.},
\newblock{\em Lebesgue decomposition of contents via nonnegative forms},
\newblock{Acta Math. Hung.}, \newblock{140(1-2) (2013)}, \newblock{151--161.}

\end{thebibliography}
\end{document}